\documentclass[runningheads,a4paper]{llncs}

\usepackage{amssymb}
\usepackage{graphicx}

\usepackage{url}
\urldef{\mailsa}\path|{alfred.hofmann, ursula.barth, ingrid.haas, frank.holzwarth,|
\urldef{\mailsb}\path|anna.kramer, leonie.kunz, christine.reiss, nicole.sator,|
\urldef{\mailsc}\path|erika.siebert-cole, peter.strasser, lncs}@springer.com|    
\newcommand{\keywords}[1]{\par\addvspace\baselineskip
\noindent\keywordname\enspace\ignorespaces#1}

\usepackage{amsmath}
\usepackage{mathrsfs}
\usepackage{mathtools}
\usepackage{etoolbox}
\usepackage{enumerate}
\usepackage{tikz}
\usetikzlibrary{arrows}

\newcommand{\ZZ}{\mathbb{Z}}			
\newcommand{\NN}{\mathbb{N}}			
\newcommand{\family}[1]{
	\mathcal{#1}%
}
\newcommand{\isdef}{\triangleq}			
\newcommand{\abs}[1]{
	\left\lvert#1\right\rvert%
}
\newcommand{\xPr}{\mathbb{P}}			
\newcommand{\diff}{
	\operatorname{\mathrm{diff}}%
}
\newcommand{\symb}[1]{\mathtt{#1}}		
\newcommand{\unif}[1]{\underline{#1}}	

\newcommand{\maj}{
	\operatorname{\mathrm{maj}}%
}

\newcommand{\critical}{\textrm{c}}							

\begin{document}

\mainmatter  

\title{%
	Restricted density classification in one dimension
}
\titlerunning{%
	Restricted density classification
}

\author{%
	Siamak Taati
}
\authorrunning{S.~Taati}

\institute{%
	Mathematics Institute, Leiden University\\
	P.O.~Box 9512, 2300 RA Leiden, The Netherlands}

\toctitle{Lecture Notes in Computer Science}
\tocauthor{Authors' Instructions}
\maketitle

\begin{abstract}
	The density classification task is to determine which of the symbols appearing
	in an array has the majority.
	A cellular automaton solving this task is required to converge to a uniform configuration
	with the majority symbol at each site.
	It is not known whether a one-dimensional cellular automaton with binary alphabet
	can classify all Bernoulli random configurations almost surely according to their densities.
	We show that any cellular automaton that washes out finite islands
	in linear time classifies all Bernoulli random configurations
	with parameters close to $0$ or $1$ almost surely correctly.
	The proof is a direct application of a ``percolation'' argument
	which goes back to G\'acs (1986).
\end{abstract}

\section{Introduction}
An array containing symbols $\symb{0}$ and $\symb{1}$ is given.
We would like to determine which of the two symbols $\symb{0}$ and $\symb{1}$
appears more often in this array.
The challenge is to perform this task in a local, uniform and decentralized fashion,
that is, by means of a cellular automaton.
A cellular automaton solving this problem is to receive the input array as its
initial configuration and to end by reaching a consensus, that is,
by turning every symbol in the array into the majority symbol.
All computations must be done on the same array with no additional symbols.

If we require the cellular automaton to solve the task for
all odd-sized finite arrays with periodic boundary conditions
(i.e., arrays indexed by a ring $\ZZ_n$ or a $d$-dimensional torus $\ZZ_n^d$,
where $n$ is odd), then no perfect solution exists~\cite{LanBel95} (see also~\cite{BusFatMaiMar13}).
Indeed, the effect of an isolated $\symb{1}$ deep inside a large region of $\symb{0}$'s
will soon disappear, hence its removal from the starting configuration
should not affect the end result.  However, removing such an isolated $\symb{1}$
could shift the balance of the majority from $\symb{1}$ to $\symb{0}$
in a borderline case.

Here, we consider a variant of the problem on infinite arrays, 
and focus on the one-dimensional case.
We ask for a cellular automaton that classifies
a randomly chosen configuration (say, using independent biased coin flips)
according to its density \emph{almost surely} (i.e., with probability~$1$).
We relax the notion of classification to allow computations that take infinitely long:
we only require that the content of each site is eventually turned into
the majority symbol and remains so forever, but we allow the fixation time
to depend on the site.

Almost sure classification of random initial configurations
is closely related to the question of stability of cellular automata trajectories against noise
and the notion of ergodicity for probabilistic cellular automata.
Constructing a cellular automaton with at least two distinct trajectories
that remain distinguishable 
in presence of positive Bernoulli noise is far from trivial.
Toom~\cite{Too74,Too80} produced a family of examples in two dimensions.
Each of Toom's cellular automata has two or more 
distinct fixed points that are stable against noise:
in presence of sufficiently small (but positive) Bernoulli noise,
the cellular automaton starting from each of these fixed points remains
close to that fixed point for an indefinite amount of time.
The noisy version of each of these cellular automata is thus non-ergodic in that it has
more than one invariant measure.

The most well-known of Toom's examples is the so-called \emph{NEC} rule
(NEC standing for North, East, Center).
The NEC rule replaces the symbol at each site
with the majority symbol among the site itself and its north and east neighbors.
Combining the combinatorial properties of the NEC rule and well-known results
from percolation theory, Bu\v{s}i\'c, Fat\`es, Mairesse and Marcovici~\cite{BusFatMaiMar13} showed that
the NEC cellular automaton also solves the classification problem:
starting from a random Bernoulli configuration with parameter $p$ on $\ZZ^2$
(i.e., using independent coin flips with probability $p$
of having $\symb{1}$ at each site),
the cellular automaton
converges almost surely to the uniform configuration $\unif{\symb{0}}$
if $p<1/2$ and to $\unif{\symb{1}}$ if $p>1/2$.

The situation in dimension one is more complicated.
No one-dimensional cellular automaton with binary alphabet
is known to classify Bernoulli random configurations.
Moreover, Toom's examples do not extend to one dimension;
the only example of a one-dimensional cellular automaton
with distinct stable trajectory in presence of noise
is a sophisticated construction due to G\'acs~\cite{Gac86,Gac01}
based on error-correction and self-simulation,
which uses a huge number of symbols per site.

There are however candidate cellular automata in one dimension
that are suspected to both classify Bernoulli configurations
and to remain bi-stable in presence of noise.
The oldest, most studied candidate is
the \emph{GKL} cellular automaton, introduced by
G\'acs, Kurdyumov and Levin~\cite{GacKurLev78}.
Another candidate with similar properties and same degree of simplicity
is the \emph{modified traffic} cellular automaton
studied by K\r{u}rka~\cite{Kur03} and Kari and Le Gloannec~\cite{KarGlo12}.
Both of these two automata have the important property that
they ``wash out finite islands of errors''
on either of the two uniform configurations~$\unif{\symb{0}}$ and $\unif{\symb{1}}$~\cite{GonMae92,KarGlo12}.
In other words, each of the two uniform configurations~$\unif{\symb{0}}$ and $\unif{\symb{1}}$
is a fixed point that attracts
all configurations that differ from it at no more than finitely many sites.
Incidentally, this same property is also shared among Toom's cellular automata,
and is crucial (but not sufficient) for
its noise stability and density classification properties.

A cellular automaton that washes out finite islands of errors,
also washes out infinite sets of errors that are sufficiently sparse.
In this context, a set should be considered sparse if it can be
covered with disjoint finite islands that are washed out before
sensing the effect of (or having an effect on) one another.
It turns out that a Bernoulli random configuration with sufficiently small parameter
is sparse with probability~$1$.
The proof is via a beautiful and relatively simple argument
that goes back to G\'acs~\cite{Gac86,Gac01},
who used it to take care of the probabilistic part of his result.
The author has learned this argument in a more streamlined form from
a recent paper of Durand, Romashchenko and Shen~\cite{DurRomShe12},
who used it in the context of aperiodic tilings.
Given its simplicity and potential, we shall repeat this argument below.

An immediate consequence of the sparseness of low-density Bernoulli sets
is that any cellular automaton that
washes out finite islands of errors on~$\unif{\symb{0}}$ and~$\unif{\symb{1}}$
(e.g., GKL and modified traffic) almost surely classifies
a Bernoulli random configuration correctly, as long as the
Bernoulli parameter $p$ is close to either $0$ or $1$.
It remains open whether the same classification occurs
for all values of $p$ in $(0,1/2)\cup(1/2,1)$.

\subsection{Terminology}
Let us proceed by fixing the terminology and formulating the problem more precisely.
By a \emph{configuration}, we shall mean an infinite array of symbols
$x_i$ chosen from an alphabet $S$
that are indexed by integers $i\in\ZZ$,
or equivalently, a function $x:\ZZ\to S$.
The evolution of a cellular automaton is obtained by
iterating a transformation
$\Phi:S^\ZZ\to S^\ZZ$
on a starting configuration $x:\ZZ\to S$.
The transformation $x\mapsto \Phi x$ is carried out
by applying a \emph{local update rule} $f$ simultaneously on every site
so that the new symbol at site $i$
reads $(\Phi x)_i\isdef f(x_{i-r},x_{i-r+1},\ldots,x_{i+r})$.
We call the sites $i-r,i-r+1,\ldots,i+r$ the \emph{neighbors} of site $i$
and refer to $r$ as the neighborhood \emph{radius} of the cellular automaton.

The \emph{density} of a symbol $a$ in a configuration $x$
is not always well-defined or non-ambiguous.
We take as the definition,
\begin{align}
	\rho_{a}(x) &\isdef
		\lim_{N\to\infty} \frac{\abs{\{i\in[-N,N]: x_i=a\}}}{2N+1}
\end{align}
when the limit exists.
According to the law of large numbers,
the density of a symbol $a$ in a Bernoulli random configuration
is almost surely the same as the probability of occurrence of $a$ at each site.
Formally, if $X$ is a random configuration $\ZZ\to S$
in which the symbol at each site is chosen independently of the others,
taking value $a$ with probability $p(a)$,
then $\xPr\{\rho_a(X)=p(a)\}=1$.

When $S=\{\symb{0},\symb{1}\}$, we simply write $\rho(x)\isdef\rho_{\symb{1}}(x)$
for the density of $\symb{1}$'s in $x$.
We say that a cellular automaton $\Phi:\{\symb{0},\symb{1}\}^\ZZ\to\{\symb{0},\symb{1}\}^\ZZ$
\emph{classifies} a configuration $x:\ZZ\to\{\symb{0},\symb{1}\}$
according to density if $\Phi^t x\to\unif{\symb{0}}$ or $\Phi^t x\to\unif{\symb{1}}$
as $t\to\infty$,
depending on whether $\rho(x)<1/2$ or $\rho(x)>1/2$.
The notation $\unif{a}$ is used to denote a uniform configuration
with symbol $a$ at each site.
For us, the meaning of the convergence of a sequence of configurations
$x^{(1)},x^{(2)},\ldots$ to another configuration $x$
is \emph{site-wise eventual agreement}: for each site $i$, there must be
an index $n_i$ after which all the following configurations in the sequence
agree with $x$ on the content of site $i$.
(Formally, $x^{(n)}_i=x_i$ for all $n\geq n_i$.)
This is the concept of convergence in the product topology of $S^\ZZ$,
which is a compact and metric topology.

\section{Eroder Property}
Let us describe two candidates that are suspected to solve
the density classification problem in one dimension:
the cellular automaton of G\'acs, Kurdyumov and Levin and the modified traffic rule.
Both cellular automata are defined
on binary configurations $\ZZ\to\{\symb{0},\symb{1}\}$
and have neighborhood radius $3$.

The cellular automaton of G\'acs, Kurdyumov and Levin~\cite{GacKurLev78} (\emph{GKL} for short)
is defined by the transformation
\begin{align}
	(\Phi x)_i &\isdef \begin{cases}
		\maj(x_{i-3}, x_{i-1}, x_i)		& \text{if $x_i=\symb{0}$,} \\
		\maj(x_i, x_{i+1}, x_{i+3})		& \text{if $x_i=\symb{1}$,}
	\end{cases}
\end{align}
where $\maj(a,b,c)$ denotes the majority symbol among $a,b,c$.

The \emph{modified traffic} cellular automaton~\cite{Kur03,KarGlo12} is defined
as a composition of two simpler automata:
the traffic automaton followed by a smoothing filter.
The \emph{traffic} automaton transforms a configuration
by replacing every occurrence of $\symb{1}\symb{0}$ with $\symb{0}\symb{1}$.
The follow-up filter replaces the $\symb{1}$ in every occurrence of $\symb{0}\symb{0}\symb{1}\symb{0}$
with~$\symb{0}$, and symmetrically,
turns the $\symb{0}$ in every occurrence of $\symb{1}\symb{0}\symb{1}\symb{1}$
into a~$\symb{1}$.

Sample space-time diagrams of the GKL and the modified traffic automata
are depicted in Figure~\ref{fig:space-time:sample}.
\begin{figure}
	\begin{center}
		\begin{tabular}{ccc}
			\begin{minipage}[c]{0.49\textwidth}
				\centering
				\includegraphics[scale=0.125]{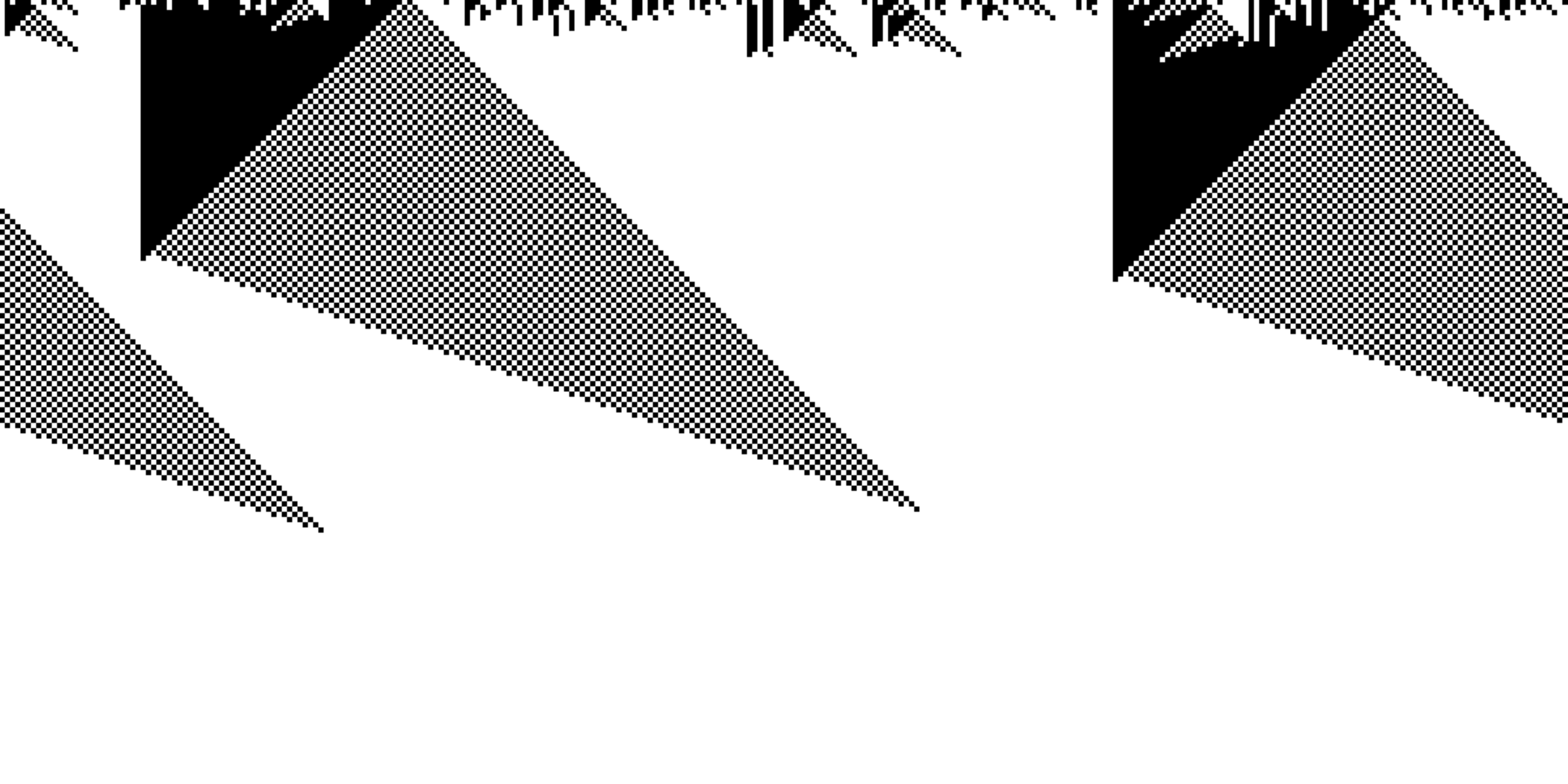}
			\end{minipage}
			& &
			\begin{minipage}[c]{0.49\textwidth}
				\centering
				\includegraphics[scale=0.125]{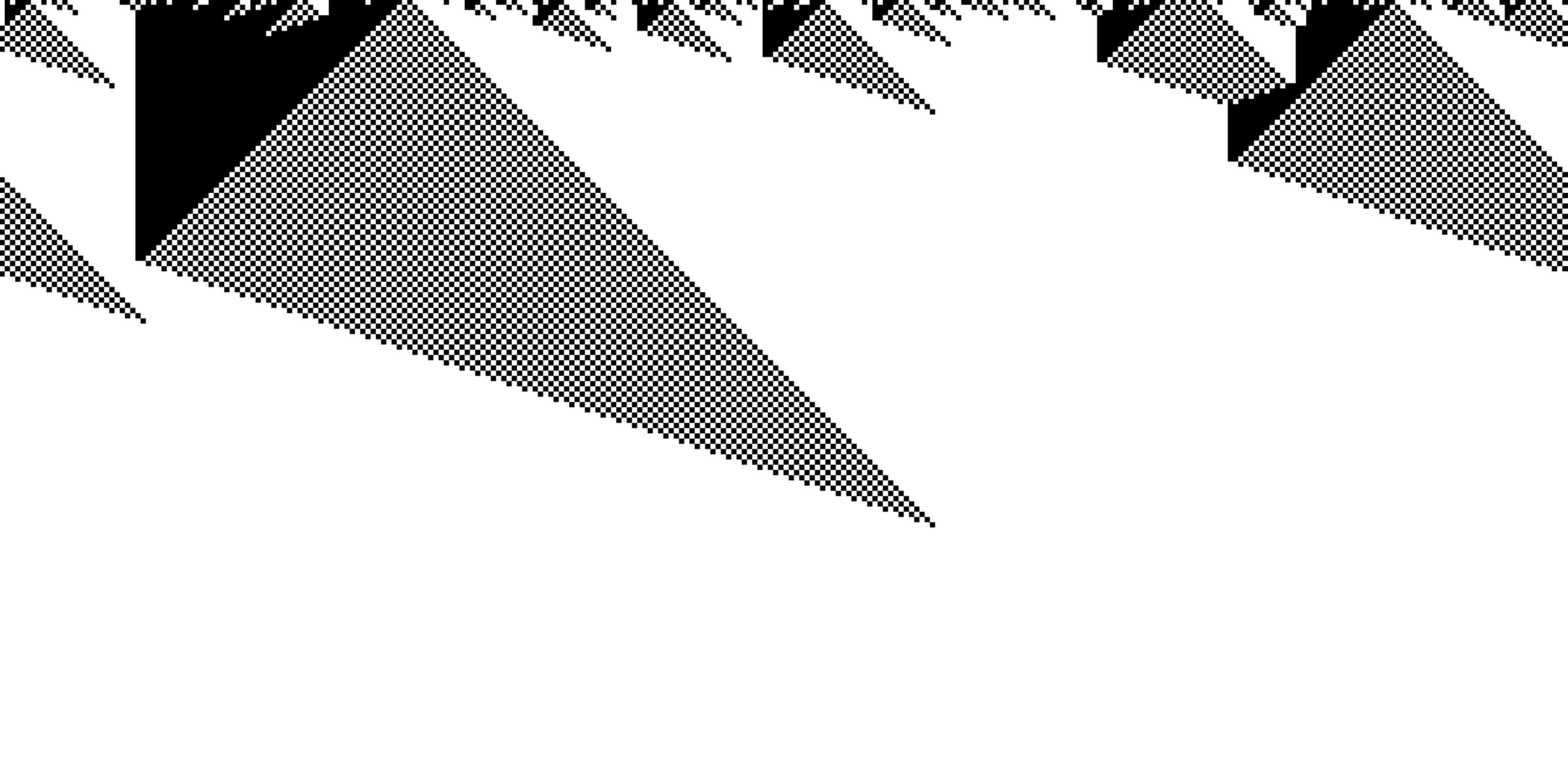}
			\end{minipage}
			\\
			(a) GKL & & (b) modified traffic
		\end{tabular}
	\end{center}
	\caption{
		Finding the majority in a biased coin-flip configuration.
		Time goes downwards.
	}
	\label{fig:space-time:sample}
\end{figure}
Note that both GKL and modified traffic
have the following symmetry: exchanging $\symb{0}$ with $\symb{1}$
and right with left leaves the cellular automaton unchanged.

The uniform configurations $\unif{\symb{0}}$ and $\unif{\symb{1}}$
are fixed points of both GKL and modified traffic automata.
The following theorem states that both automata
wash out finite islands of errors
on either of the two uniform configurations $\unif{\symb{0}}$ and $\unif{\symb{1}}$.
This is sometimes called the \emph{eroder property}. 
For the GKL automaton, the eroder property was proved by Gonzaga de~S\'a and Maes~\cite{GonMae92};
for modified traffic, the result is due to Kari and Le~Gloannec~\cite{KarGlo12}.
Let us write $\diff(x,y)\isdef \{i\in\ZZ: x_i\neq y_i\}$
for the set of sites at which two configurations $x$ and $y$ differ.
We call $x$ a \emph{finite perturbation} of $z$ if $\diff(z,x)$ is a finite set.
\begin{theorem}[Eroder property~\cite{GonMae92,KarGlo12}]
	Let $\Phi$
	be either the GKL or the modified traffic cellular automaton.
	For every finite perturbation $x$ of $\unif{\symb{0}}$, there is a time $t$
	such that $\Phi^t x=\unif{\symb{0}}$.
	If $\diff(\unif{\symb{0}},x)$ has diameter at most $n$
	(i.e., covered by an interval of length $n$),
	then $\Phi^{2n} x = \unif{\symb{0}}$.
	The analogous statement about finite perturbations of $\unif{\symb{1}}$
	holds by symmetry.
\end{theorem}

Let us emphasize that many simple cellular automata
have the eroder property on some uniform configuration.
For instance, the cellular automaton $\Phi:\{\symb{0},\symb{1}\}^\ZZ\to \{\symb{0},\symb{1}\}^\ZZ$
defined by $(\Phi x)_i\isdef x_{i-1} \land x_i \land x_{i+1}$
washes out finite islands on the uniform configuration $\unif{\symb{0}}$.
What is remarkable about GKL and modified traffic
is the fact that they have the eroder property on
\emph{two distinct} configurations $\unif{\symb{0}}$ and~$\unif{\symb{1}}$.
This double eroder property may lead one to guess that
these two cellular automata could indeed classify Bernoulli configurations
according to density or that the trajectories of
the fixed points $\unif{\symb{0}}$ and $\unif{\symb{1}}$
are stable in presence of small but positive noise.

\section{Washing Out Sparse Sets}
In this section, we consider a slightly more general setting.
We assume that $\Phi:S^\ZZ\to S^\ZZ$ is
a cellular automaton that washes out finite islands of errors
on a configuration $z$ in linear time;
that is, there is a constant $m$ such that
$\Phi^{ml} x = \Phi^{ml} z$ for any finite perturbation of $x$
for which $\diff(z,x)$ has diameter at most~$l$.
For GKL and modified traffic, $z$ can be either $\unif{\symb{0}}$ or $\unif{\symb{1}}$,
which are fixed points (hence $\Phi^{ml} z=z$),
and the constant $m$ can be chosen to be $2$.

The above eroder property automatically implies
that $\Phi$ also washes out (possibly infinite) sets of error
that are ``sparse enough''.  Indeed, an island of errors
which is well separated from the rest of the errors will
disappear before sensing or affecting the rest of the error set. 
We are interested in an appropriate notion of ``sparseness''
for $\diff(z,x)$ that guarantees the attraction of the trajectory of $x$
towards the trajectory of $z$.

To elaborate this further, let us denote the neighborhood radius of $\Phi$ by~$r$.
Consider an arbitrary configuration $x$ and think of it as a perturbation of $z$
with errors occurring at sites in $\diff(z,x)$.
Let $I\subseteq\ZZ$ be an interval of length $l$
such that $x$ agrees with $z$ on a margin of width $2rml$ around $I$,
that is, $x_j=z_j$ for $j\in\ZZ\setminus I$ within distance $2rml$ from $I$.
We call such an interval an \emph{isolated island} (of errors) on $x$.
Let $y$ be a configuration obtained from $x$
by \emph{erasing} the errors on $I$, that is, by replacing $x_i$ with $z_i$
for each $i\in I$.
(Note on terminology:
we shall use ``erasure'' to refer to this abstract construction of one configuration
from another, and reserve the word ``washing'' for what the cellular automaton does.)
Observe that within $ml$ steps, the distinction between $x$ and $y$ disappears
and we have $\Phi^{ml}x=\Phi^{ml}y$ (see Figure~\ref{fig:washing:isolated}).%
\begin{figure}
	\begin{center}
		\begin{tikzpicture}[>=stealth',shorten >=0.5pt,shorten <=0.5pt]
			\useasboundingbox (-5,-1.5) rectangle (5,1);
			\fill[fill=gray!10] (-4,0) -- (-2.5,-1.5) -- (2.5,-1.5) -- (4,0) -- cycle;
			\draw[help lines] (-4,0) -- (-2.5,-1.5) -- (2.5,-1.5) -- (4,0) -- cycle;
			\fill[fill=gray!60] (-1,0) -- (1,0) -- (2.5,-1.5) -- (-2.5,-1.5) -- cycle;
			\draw[help lines] (-5,0) -- (5,0);
			\draw[ultra thick, line cap=round] (-1,0) -- (1,0);
			\draw[<->,very thin] (-4.3,0) -- (-4.3,-1.5) node[midway,left] {$ml$};
			\draw[<->,very thin] (-1,0.2) -- (1,0.2) node[midway,above] {$l$};
			\draw[<->,very thin] (-4,0.2) -- (-2.5,0.2) node[midway,above] {$rml$};
			\draw[<->,very thin] (-2.5,0.2) -- (-1,0.2) node[midway,above] {$rml$};
			\draw[<->,very thin] (1,0.2) -- (2.5,0.2) node[midway,above] {$rml$};
			\draw[<->,very thin] (2.5,0.2) -- (4,0.2) node[midway,above] {$rml$};
		\end{tikzpicture}
	\end{center}
	\caption{
		Forgetting an isolated region of errors.
	}
	\label{fig:washing:isolated}
\end{figure}
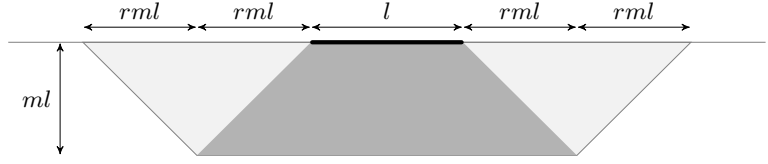
Namely, the island $I$ is washed out before time $ml$
and the sites in $\diff(z,x)\setminus I=\diff(z,y)\setminus I$ do not get a chance
to feel the distinction between $x$ and $y$.

We find that erasing an isolated island of length at most $l$
from $x$ does not affect whether the trajectory of $x$ is attracted towards
the trajectory of $z$ or not.
Neither does erasing several (possibly infinitely many) isolated islands of length $\leq l$
at the same time.
On the other hand, erasing some isolated islands from $x$
makes the error set $\diff(z,x)$ sparser and possibly turns
larger portions of $\diff(z,x)$ into isolated islands
(see Figure~\ref{fig:washing:sparse}).%
\begin{figure}
	\begin{center}
		\begin{tikzpicture} 
			\clip (-6,-2) rectangle (6,0.5);
			
			\fill[fill=gray!60] (-4,0) -- (-5.5,-1.5) -- (-0.5,-1.5) -- (-2,0) -- cycle;
			\fill[fill=gray!60] (1.5,0) -- (-0.375,-1.875) -- (6.375,-1.875) -- (4.5,0) -- cycle;
			\fill[fill=gray!60] (-0.25,0) -- (-0.625,-0.375) -- (0.625,-0.375) -- (0.25,0) -- cycle;
			\fill[fill=gray!60] (-5.8,0) -- (-6.25,-0.45) -- (-4.75,-0.45) -- (-5.2,0) -- cycle;
			
			\draw[help lines] (-6,0) -- (6,0);
			\draw[ultra thick, line cap=round] (-4,0) -- (-2,0);
			\draw[ultra thick, line cap=round] (1.5,0) -- (4.5,0);
			\draw[ultra thick, line cap=round] (-0.25,0) -- (0.25,0);
			\draw[ultra thick, line cap=round] (-5.8,0) -- (-5.2,0);
		\end{tikzpicture}
	\end{center}
	\caption{
		Washing out a sparse set of errors.
	}
	\label{fig:washing:sparse}
\end{figure}
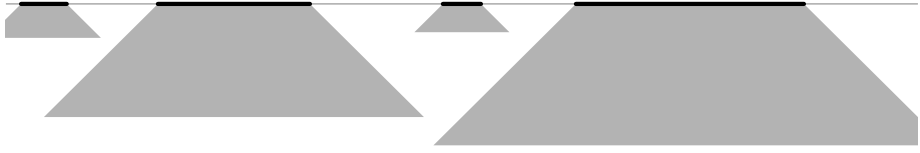
Hence, we can perform the erasure procedure recursively,
by first erasing the isolated islands of length $1$,
then erasing the isolated islands of length $2$,
then erasing the isolated islands of length $3$ and so forth.
In this fashion, we obtain a sequence
$x^{(0)}, x^{(1)}, x^{(2)}, \ldots$
with $x^{(0)}=x$ and $\diff(z,x^{(l)})\supseteq \diff(z,x^{(l+1)})$
obtained by successive erasure of isolated islands.
We say that the error set $\diff(z,x)$ is \emph{sparse}
if all errors are eventually erased, that is,
if $\bigcap_l \diff(z,x^{(l)})=\varnothing$.

However, this notion of sparseness still does not guarantee
the attraction of the trajectory of $x$ towards the trajectory of $z$.
(The trajectory of $x$ is considered to be \emph{attracted} towards the trajectory of $z$
if for each site $i$, there is a time $t_i$ such that
$(\Phi^t x)_i=(\Phi^t z)_i$ for all time steps $t\geq t_i$.
If $\Phi z=z$, this attraction becomes equivalent to the convergence $\Phi^t x\to z$.)
Note that it is well possible that all errors are eventually washed out
from $x$ (hence, their information is lost) but the washing out procedure
for larger and larger islands affects a given site $i$ indefinitely,
so that $(\Phi^t x)_i\neq(\Phi^t z)_i$ for infinite many time steps $t$
(see Figure~\ref{fig:washing:non-attracting}).%
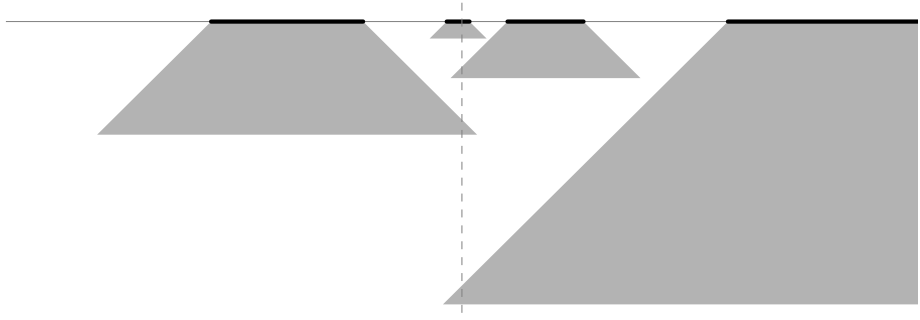
\begin{figure}
	\begin{center}
		\begin{tikzpicture} 
			\clip (-6,-4.25) rectangle (6,0.5);
			
			\fill[fill=gray!60] (-0.2,0) -- (-0.425,-0.225) -- (0.325,-0.225) -- (0.1,0) -- cycle;
			\fill[fill=gray!60] (0.6,0) -- (-0.15,-0.75) -- (2.35,-0.75) -- (1.6,0) -- cycle;
			\fill[fill=gray!60] (-3.3,0) -- (-4.8,-1.5) -- (0.2,-1.5) -- (-1.3,0) -- cycle;
			\fill[fill=gray!60] (3.5,0) -- (-0.25,-3.75) -- (12.25,-3.75) -- (8.5,0) -- cycle;
			
			\draw[help lines] (-6,0) -- (6,0);
			\draw[ultra thick, line cap=round] (-0.2,0) -- (0.1,0);
			\draw[ultra thick, line cap=round] (0.6,0) -- (1.6,0);
			\draw[ultra thick, line cap=round] (-3.3,0) -- (-1.3,0);
			\draw[ultra thick, line cap=round] (3.5,0) -- (8.5,0);
			
			\draw[dashed, help lines] (0,0.25) -- (0,-4);
		\end{tikzpicture}
	\end{center}
	\caption{
		Washing out but not attracting.
	}
	\label{fig:washing:non-attracting}
\end{figure}

To clarify this possibility,
note that an isolated island of length $l$ can affect
the state of sites within distance $rml$ up to time $ml$
(see Figure~\ref{fig:washing:isolated}).
Let us denote by $A_l\isdef\diff(z,x^{(l)})\setminus\diff(z,x^{(l-1)})$
the union of isolated islands of length $l$ that are erased from $x^{(l-1)}$
during the $l$'th stage of the erasure procedure.
The only possibility for a site $i$ to have a value other than $(\Phi^t z)_i$
at time $t$ is that site $i$ is within distance $rml$ from
$A_l$ for some $l$ satisfying $ml>t$.
In this case, we say that $i$ is 
within the \emph{territory} of such $A_l$.
A sufficient condition for the attraction of the trajectory of $x$
towards the trajectory of $z$ is that
the error set $\diff(z,x)$ is sparse, and furthermore, each site $i$ is 
within the territory of
$A_l$ for at most finitely many values of $l$.
If this condition is satisfied, we say that the error set $\diff(z,x)$ is \emph{strongly sparse}.
In summary, the trajectory of $x$ is attracted towards the trajectory of $z$
if $\diff(z,x)$ is \emph{strongly sparse}.

\section{Sparseness}
The notion of (strong) sparseness described in the previous section
can be formulated and studied without reference to cellular automata,
and that is what we are going to do now.
This notion is of independent interest,
as it commonly arises in error correcting scenarios.
More sophisticated applications appear in~\cite{Gac86,Gac01} and~\cite{DurRomShe12}.
Our exposition is close to that of~\cite{DurRomShe12}.

We refer to a finite interval $I\subseteq\ZZ$ as an \emph{island}.
Let $k$ be a fixed positive integer.
The \emph{territory} (or the \emph{interaction range}) of an island $I$ of length $l$
is the set of sites $i\in\ZZ$ that are within distance $kl$ from $I$.
We denote the territory of $I$ by $R(I)$.
Two disjoint islands $I$ and $I'$ of lengths $l$ and $l'$, where $l\leq l'$,
are considered \emph{well separated} if $I'\cap R(I)=\varnothing$,
that is, 
if the larger island does not intrude the territory of the smaller one.
A set $E\subseteq\ZZ$ is said to be \emph{sparse}
if it can be covered by a family $\family{I}$
of (disjoint) pairwise well-separated islands.
A sparse set is \emph{strongly sparse} if the cover $\family{I}$
can be chosen so that each site $i$ is in the territory of
at most finitely many elements of $\family{I}$.

Note that for $k\isdef 2rm$, 
we get essentially the same notion of sparseness as in the previous section.
Indeed, let $\family{I}_l$ be the sub-family of $\family{I}$
containing all islands of length at most $l$,
and denote by $E_l\isdef E\setminus\bigcup_{I\in\family{I}_l}I$
the subset of $E$ obtained by erasing the islands of length at most~$l$.
Then, every island $I\in\family{I}$ having length~$l$
is \emph{isolated} in $E_{l-1}$, because its territory is not intruded by $E_{l-1}\setminus I$.
The new notion of strong sparseness might be slightly more restrictive,
as we define the territory by the constant $k=2rm$ rather than $k/2=rm$,
but the arguments below are not sensitive to this distinction.

The most basic observation about sparseness is
its monotonicity.
\begin{proposition}[Monotonicity]
	Any subset of a (strongly) sparse set is (strongly) sparse.
\end{proposition}

One expects a ``small'' set to be sparse.
The following theorem due to Levin~\cite{Lev00} is an indication of
this intuition.
\begin{theorem}[Sparseness of small sets~\cite{Lev00}]
	There are constants $\varepsilon,c\in(0,1)$ depending on the sparseness parameter $k$
	such that every periodic set $E\subseteq\ZZ$ with period $n$
	and at most $c\, n^\varepsilon$ elements per period is strongly sparse.
\end{theorem}

The reverse intuition is misleading:
a sparse set does not need to be ``small''.
In fact, there are sets with arbitrarily large densities that are sparse.
The existence of such sets is demonstrated by Kari and Le~Gloannec~\cite{KarGlo12},
and in special cases, was also noted
by Levin~\cite{Lev00} and K\r{u}rka~\cite{Kur03}.
\begin{theorem}[Large sparse sets~\cite{KarGlo12}]
	There are periodic subsets of $\ZZ$ with density arbitrarily close to $1$
	that are strongly sparse.
\end{theorem}
It immediately follows that the set of possible densities of
strongly sparse (periodic) subsets of $\ZZ$ is dense in $[0,1]$.
A more important corollary is a strengthening of the impossibility result of
Land and Belew~\cite{LanBel95} for cellular automata with \emph{linear-time} eroder property:
for any such automaton, there are configurations $x$ with density $\rho(x)$
close to any number in $[0,1]$ that are incorrectly classified.

The main result of interest for us is the sparseness of
sufficiently biased Bernoulli random sets.
\begin{theorem}[Sparseness of Bernoulli sets~\cite{Gac86,Gac01,DurRomShe12}]
\label{thm:sparseness:Bernoulli}
	A Bernoulli random set $E\subseteq\ZZ$ with parameter $p$
	is almost surely strongly sparse
	as long as $p<(2k)^{-2}$, 
	where $k$ is the sparseness parameter.
\end{theorem}
\begin{proof}
	For a set $E\subseteq\ZZ$, we recursively construct a family $\family{I}$
	of pairwise well-separated islands as a candidate for covering $E$.
	The family $\family{I}$ will be divided into sub-families $\family{J}_l$
	consisting of islands of length $l$,
	and $E_l$ will be the set 
	obtained by erasing 
	the selected islands of length at most $l$ from $E$.
	Let $E_0\isdef E$.
	For $l\geq 1$, recursively define $\family{J}_l$ as the family of islands $I\subseteq\ZZ$
	of length $l$ that intersect $E_{l-1}$ and are isolated in $E_{l-1}$
	(i.e., $E_{l-1}\setminus I$ does not intersect the territory of $I$),
	and set $E_l\isdef E_{l-1}\setminus\bigcup_{I\in\family{J}_l}I$.
	Let $\family{I}\isdef\bigcup_l\family{J}_l$.
	
	To see that the elements of $\family{I}$ are pairwise well separated,
	let us first argue that every island $I\in\family{J}_l$ is minimal,
	in that, it is the smallest interval containing $I\cap E_{l-1}$.
	Indeed, let $J\subseteq I$ be the smallest island containing $I\cap E_{l-1}$,
	and assume that $\abs{J}<l$.
	Then, the endpoints of $J$ must be in $E_{l-1}$.
	Therefore, every island $I'\in\family{J}_{l'}$ with $l'<l$
	must have been at distance more than $kl'$ from $J$,
	for otherwise, $I'$ would not have been isolated in $E_{l'-1}$.
	In particular, for $l'$ satisfying $\abs{J}\leq l' <l$,
	the island $J$ has distance more than $k\abs{J}$ from
	every $I'\in\family{J}_{l'}$.
	Since the distance between $J$ and $E_{l-1}$ is also more than $kl\geq k\abs{J}$,
	it follows that $J$ is isolated in $E_{\abs{J}-1}$.
	On the other hand, $J$ intersects $E_{\abs{J}-1}$,
	because it intersects $E_{l-1}$ and $E_{l-1}\subseteq E_{\abs{J}-1}$.
	We find that $J\in\family{J}_{\abs{J}-1}$,
	which is a contradiction.
	Therefore, $I$ is minimal.
	The well-separation of two islands $I\in\family{J}_l$ and $I'\in\family{J}_{l'}$
	with $l\leq l'$ follows from the minimality of $I'$.
	We conclude that the elements of $\family{I}$ are also well separated.
	
	Now, let $E$ be a Bernoulli random configuration with parameter $p$.
	We choose an appropriate sequence $0<l_1<l_2<l_3<\cdots$
	(to be specified more explicitly below)
	and observe whether a site $u$ is in $E_{l_n}$.
	We will show that the probability that site $u$ is in $E_{l_n}$
	is double exponentially small, that is,
	$\xPr(u\in E_{l_n})\leq\alpha^{2^n}$ for some $\alpha<1$.
	
	Let $u$ be an arbitrary site.
	In order for $u$ to be in $E_{l_n}$,
	it is necessary that $u$ is also in $E_{l_n-1}$,
	and furthermore, $u$ is not covered by any island in $\family{J}_{l_n}$.
	Therefore, $E_{l_{n-1}}$ (which includes $E_{l_n-1}$) must contain two elements
	$u_{\symb{0}}\isdef u$ and $u_{\symb{1}}$ that
	are farther than $l_n/2$ from each other
	but no farther than $(k+1/2)l_n$ from each other
	(see Figure~\ref{fig:explanation-tree:children}).
    \begin{figure}
    	\begin{center}
    		\begin{tikzpicture}[>=stealth',shorten >=0.5pt,shorten <=0.5pt]
    			
    			\draw[help lines] (-4,0) -- (4,0);
    			
    			\fill (0,0) circle (2pt) node[below=1ex] {$u_{\symb{0}}=u$};
    			\fill (3,0) circle (2pt) node[below=1ex] {$u_{\symb{1}}$};

    			\foreach \x in {-4, -1, 1, 4}
    				\draw[help lines] (\x,0.1) -- (\x,-0.1);
    			
    			\draw[->,very thin] (0,0.3) -- (1,0.3) node[midway,above] {$l_n/2$};
    			\draw[->,very thin] (0,0.3) -- (-1,0.3) node[midway,above] {$l_n/2$};
    			\draw[<->,very thin] (-4,0.3) -- (-1,0.3) node[midway,above] {$kl_n$};
    			\draw[<->,very thin] (1,0.3) -- (4,0.3) node[midway,above] {$kl_n$};
    			
    		\end{tikzpicture}
    	\end{center}
    	\caption{
    		An evidence for $u\in E_{l_n}$ in $E_{l_{n-1}}$
    		(see the proof of Theorem~\ref{thm:sparseness:Bernoulli}).
    	}
    	\label{fig:explanation-tree:children}
    \end{figure}
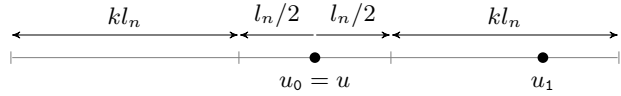
	In a similar fashion, in order for $u_{\symb{0}}$ and $u_{\symb{1}}$
	to be in $E_{l_{n-1}}$,
	the set $E_{l_{n-2}}$ must contain
	elements $u_{\symb{0}\symb{0}}\isdef u_{\symb{0}}$,
	$u_{\symb{0}\symb{1}}$, $u_{\symb{1}\symb{0}}\isdef u_{\symb{1}}$ and $u_{\symb{1}\symb{1}}$
	such that
	\begin{align}
		\frac{1}{2}l_{n-1} &< d(u_{\symb{0}\symb{0}},u_{\symb{0}\symb{1}})
			\leq \Big(k+\frac{1}{2}\Big)l_{n-1} \;,\\
		\frac{1}{2}l_{n-1} &< d(u_{\symb{1}\symb{0}},u_{\symb{1}\symb{1}})
			\leq \Big(k+\frac{1}{2}\Big)l_{n-1} \;.
	\end{align}
	Repeating this procedure, we find a binary tree of depth $n$
	with roots in $E_0=E$ that provides an evidence
	for the presence of $u$ in $E_{l_n}$.
	We call such a tree an \emph{explanation tree}.
	Thus, in order to have $u\in E_{l_n}$,
	there must be at least one explanation tree for it.
	
	We estimate the probability of the existence of an explanation tree
	for $u\in E_{l_n}$.
	Let $T=(u,u_{\symb{0}},u_{\symb{1}}, u_{\symb{0}\symb{0}}, u_{\symb{0}\symb{1}}, \ldots,
	u_{\symb{1}\symb{1}\cdots\symb{0}},u_{\symb{1}\symb{1}\cdots\symb{1}})$
	be a \emph{candidate} explanation tree,
	that is, a tree with the right distances between the nodes.
	To simplify the estimation, we choose the
	lengths $l_1,l_2,\ldots$ in such a way to make sure that
	the leaves of $T$ are distinct elements of $\ZZ$.
	A sufficient condition for the distinctness of the leaves of $T$
	is that for each $m$,
	\begin{align}
		\frac{1}{2}l_m &\geq
			2\Big(k+\frac{1}{2}\Big)(l_{m-1}+l_{m-2}+\cdots+l_1) \;.
	\end{align}
	This would guarantee that the two subtrees
	descending from each node do not intersect.
	We choose $l_m\isdef (4k+3)^{m-1}$, which is a solution of the above system of inequalities.
	
	A candidate tree $T$ is an explanation tree for $u\in E_{l_n}$
	if and only if all its leaves are in $E$.
	Whether or not each leaf $u_w$ of $T$ is in $E$
	is determined by a biased coin flip with probability $p$ of falling in $E$.
	With the above choice of $l_m$,
	the events $u_w\in E$ for different leaves of $T$ are independent.
	It follows that $T$ is an explanation tree for $u\in E_{l_n}$
	with probability $p^{2^n}$.
	
	Let us now estimate the number of candidate trees of depth $m$.
	Denote this number by $f_m$.
	Observe that $f_m$ satisfies the recursive inequality
	\begin{align}
		f_m &\leq 2k l_m\, f_{m-1}^2
	\end{align}
	with $f_0\isdef 1$.
	Indeed, $2k l_m$ counts for the number of possible positions for
	$u_{\symb{1}}$ and $f_{m-1}^2$ counts the number of
	possibilities for each of the two subtrees.
	Letting $g_m\isdef\log f_m$, we have
	\begin{align}
		g_m &\leq a\,m + b + 2g_{m-1} \;,
	\end{align}
	where $a\isdef\log(4k+3)$ and $b\isdef\log 2k-\log(4k+3)$. 
	Expanding the last recursion we get
	\begin{align}
		g_m &\leq
			2^m(2b + a\sum_{i=0}^m\frac{i}{2^i}) \\
		&\leq
			2^m(2b + a\sum_{i=0}^\infty\frac{i}{2^i}) \\
		&=
			2^{m+1}(a+b) \;.
	\end{align}
	Therefore,
	\begin{align}
		f_m &\leq
			(2k)^{2^{m+1}} \;.
	\end{align}
	
	By the sub-additivity of the probabilities,
	we find that the probability of the existence of at least one explanation tree
	for $u\in E_{l_n}$ satisfies
	\begin{align}
		\xPr(u\in E_{l_n}) &\leq
			p^{2^n} f_n \leq \alpha^{2^n} \;,
	\end{align}
	where 
	$\alpha\isdef p(2k)^2$.
	Since 
	$p<(2k)^{-2}$, we get $\alpha<1$.
	
	The probability that a given site $u\in\ZZ$
	is in $E$ but is not covered by $\family{I}$
	(i.e., never erased)
	is
	\begin{align}
		\xPr(u\in \bigcap_l E_l) &=
			\xPr(u\in \bigcap_n E_{l_n}) =
			\lim_{n\to\infty} \xPr(u\in E_{l_n}) =
			\lim_{n\to\infty} \alpha^{2^n} = 0 \;.
	\end{align}
	Since $\ZZ$ is countable, we find, by sub-additivity,
	that $\xPr(\bigcap_l E_l\neq\varnothing)=0$,
	which means, $E$ is sparse with probability~$1$.
	
	That $E$ is strongly sparse with probability~$1$ follows by
	the Borel-Cantelli argument.
	Namely, the event that a site $u$ is in the territory of
	infinitely many islands $I\in\family{I}$
	can be expressed as
	$\bigcap_m\bigcup_{n\geq m}\{d(u,E_{l_n})\leq kl_n\}$.
	(Note that an island covering a site in $E_{l_n}$
	has length greater than $l_n$.)
	The probability that $u$ is within distance $kl_n$ from $E_{l_n}$
	satisfies
	\begin{align}
		\xPr\big(d(u,E_{l_n})\leq kl_n\big) &\leq
			(2k l_n+1)\alpha^{2^n} = (2k(4k+3)^{n-1}+1)\alpha^{2^n} \;.
	\end{align}
	Therefore,
	\begin{align}
		\xPr\Big(\bigcup_{n\geq m}\{d(u,E_{l_n})\leq kl_n\}\Big) &\leq
			\sum_{n\geq m} (2k(4k+3)^{n-1}+1)\alpha^{2^n}
		<
			\infty \;.
	\end{align}
	It follows that
	\begin{align}
		\xPr\Big(\bigcap_m\bigcup_{n\geq m}\{d(u,E_{l_n})\leq kl_n\}\Big) &\leq
			\lim_{m\to\infty} \sum_{n\geq m} (2k(4k+3)^{n-1}+1)\alpha^{2^n} = 0 \;.
	\end{align}
	Using again the countability of $\ZZ$,
	we find that, with probability~$1$,
	no site $u$ is in the territory of more than finitely may islands $I\in\family{I}$.
	That is, $E$ is almost surely strongly sparse.
	\qed
\end{proof}

Theorem~\ref{thm:sparseness:Bernoulli},
along with a standard application of monotonicity,
shows that when the Bernoulli parameter is varied,
a non-trivial phase transition occurs.
\begin{corollary}[Phase transition]
\label{cor:sparseness:Bernoulli:phase-transition}
	There is a critical value $p_\critical\in(0,1]$
	depending on the sparseness parameter $k$ such that
	a Bernoulli random set $E\subseteq\ZZ$ with parameter $p$
	is almost surely strongly sparse if $p<p_\critical$
	and is almost surely not strongly sparse if $p>p_\critical$.
\end{corollary}
\begin{proof}
	First, observe that the (strong) sparseness of $E$ is a translation-invariant event
	(i.e., for $a\in\ZZ$, the sparseness of $a+E$ is equivalent to the sparseness of $E$).
	Therefore, by ergodicity, the probability that a Bernoulli random set
	is (strongly) sparse is either~$0$ or~$1$.
	
	The presence of a threshold value $p_\critical\in[0,1]$ (possibly $0$)
	is a standard consequence of monotonicity.
	Indeed, let $U_i, i\in\ZZ$ be a collection of independent random variables
	with uniform distribution on the real interval $[0,1]$.
	For $p\in [0,1]$, define a set $E^{(p)}\isdef\{i\in\ZZ: U_i<p\}$.
	Then, $E^{(p)}$ is a Bernoulli random set with parameter $p$,
	and the collection of sets $E^{(p)}$ is increasing in $p$.
	Let $p_\critical\isdef\sup\{p: \text{$E^{(p)}$ is almost surely (strongly) sparse}\}$.
	By monotonicity, the set $E^{(p)}$ is almost surely (strongly) sparse for $p<p_\critical$
	and is almost surely not (strongly) sparse for $p>p_\critical$.
	
	Finally, we know from Theorem~\ref{thm:sparseness:Bernoulli} that $p_\critical>0$.
	\qed
\end{proof}

\section{Restricted Classification}
Let us state the claimed result of this paper explicitly
as a corollary of Theorem~\ref{thm:sparseness:Bernoulli}
and the discussions in the previous sections.
\begin{corollary}[Restricted classification]
\label{cor:classification:restricted}
	Let $\Phi:\{\symb{0}, \symb{1}\}^\ZZ\to\{\symb{0}, \symb{1}\}^\ZZ$
	be a cellular automaton that washes out finite islands of errors
	on either of the two uniform configurations $\unif{\symb{0}}$ and $\unif{\symb{1}}$
	in linear time.
	Namely, suppose that there is a constant $m$ such that
	for every finite perturbation $x$ of $\unif{\symb{0}}$ 
	for which $\diff(\unif{\symb{0}},x)$ has diameter at most $l$,
	we have $\Phi^{ml}x=\unif{\symb{0}}$,
	and similarly for $\unif{\symb{1}}$.
	Then, there is a constant $p_\critical\in(0,1/2]$
	such that $\Phi$ classifies a Bernoulli random configuration
	with parameter 
	$p\in [0,p_\critical)\cup(1-p_\critical,1]$
	almost surely correctly.
\end{corollary}

For GKL and modified traffic,
we have $k=2rm=12$.
Therefore, Theorem~\ref{thm:sparseness:Bernoulli}
only guarantees correct classification
if the Bernoulli parameter $p$ is within distance
$(2k)^{-2}=24^{-2}\approx 0.0017$
from either $0$ or $1$.


\section{Discussion}
We conclude with few comments and questions.

Corollary~\ref{cor:classification:restricted} shows that
the asymptotic behaviour of the GKL and modified traffic automata
starting from a Bernoulli random configuration undergoes
a phase transition: the cellular automaton converges to $\unif{\symb{0}}$
for $p$ close to $0$ and to $\unif{\symb{1}}$ for $p$ close to $1$.
It remains open whether the transition occurs
precisely at $p=1/2$, or if there are other transitions in between.
The result of Bu\v{s}i\'c et al.{}~\cite{BusFatMaiMar13}
shows that the transition in the NEC cellular automaton
is unique and happens precisely at $p=1/2$.


Another open issue is the behaviour of the GKL and modified traffic automata
on random configurations with non-Bernoulli distributions.
One might expect the sparseness argument to extend to measures
that are sufficiently mixing.
For instance, it should be possible to show the same kind of
classification on a Markov random configuration
that has density close to $0$ or~$1$.

It would also be interesting to see if the sparseness method can be applied to
probabilistic cellular automata that are suggested for the density classification task.
Fat\`es~\cite{Fat13} has introduced a parametric family of one-dimensional
probabilistic cellular automaton with a density classification property:
for every $n\in\NN$ and $\varepsilon>0$, there is a setting of the parameter
such that the automaton classifies
a periodic configuration with period $n$ with probability at least $1-\varepsilon$.
Does the majority-traffic rule of Fat\`es with a fixed parameter
classify sufficiently biased Bernoulli random configurations?
A two-dimensional candidate
would be the noisy version of the nearest-neighbor majority rule,
in which the noise occurs only when there is no consensus in the neighborhood.

Finally, given its various applications, one might try to
study the notion of sparseness in a more systematic fashion,
trying to capture more details about the transition.
It is curious that the notion of sparseness of Bernoulli random sets
supports a hierarchy of phase transitions, even in one dimension
where the standard notion of percolation fails.

\subsubsection*{Acknowledgments.}
Research supported by ERC Advanced Grant 267356-VARIS of Frank den Hollander.
I would like to thank Jarkko Kari
for suggesting this problem and for discussions that lead to this paper.


\bibliographystyle{splncs03}
\bibliography{bibliography}

%
%
%

\end{document}